\documentclass{article}

\usepackage{fullpage}
\usepackage{amsmath}
\usepackage{amsthm}
\usepackage{amssymb}
\usepackage{graphicx}

\newcommand{\CC}{\mathbf C}
\newcommand{\sgn}{\operatorname{sgn}}

\newtheorem{thm}{Theorem}[section]
\newtheorem{prop}[thm]{Proposition}
\newtheorem{cor}[thm]{Corollary}
\newtheorem{lemma}[thm]{Lemma}

\theoremstyle{definition}

\newcounter{saveenumi}

\title{An algorithmic Littlewood-Richardson rule}
\author{Ricky Ini Liu\\Massachusetts Institute of Technology\\Cambridge, Massachusetts\\\texttt{riliu@math.mit.edu}}

\begin{document}
\maketitle

\begin{abstract}
We introduce a Littlewood-Richardson rule based on an algorithmic deformation of skew Young diagrams and present a bijection with the classical rule. The result is a direct combinatorial interpretation and proof of the geometric rule presented by Coskun \cite{Coskun}. We also present a corollary regarding the Specht modules of the intermediate diagrams.
\end{abstract}

\section{Introduction}

The ubiquitous Littlewood-Richardson coefficients are a collection of nonnegative integer constants that appear, among other places, in algebraic combinatorics, algebraic geometry, and representation theory. A number of combinatorial rules, typically called Littlewood-Richardson rules, have been formulated to compute these constants (for instance, see \cite{KnutsonTao}, \cite{Sagan}, \cite{EC2}, \cite{Vakil}). One such rule, described by Coskun \cite{Coskun}, calculates these coefficients via the cohomology class of the intersection of two Schubert varieties in the Grassmannian: using an appropriate sequence of degenerations, the intersection is deformed into a union of smaller varieties whose Schubert classes are evident. This rule is described combinatorially by deforming certain diagrams of squares called Mondrian tableaux.

In this paper, we present a combinatorial simplification of this rule from deforming Mondrian tableaux to deforming Young diagrams. We also provide a non-geometric explanation for its ability to correctly compute Littlewood-Richardson coefficients, relating it to the classical Littlewood-Richardson rule. We also prove an easy corollary about the structure of Specht modules for the deformed diagrams.

After discussing some preliminaries in Section 2, we will introduce the rule via Algorithm 1 at the beginning of Section 3. The remainder of Section 3 will be devoted to exhibiting a bijection between this rule and a classical Littlewood-Richardson rule. In Section 4, we will give an application of this result to Specht modules. Finally, in Section 5, we will pose some directions for further research.

\section{Preliminaries}

A \emph{partition} $\lambda = (\lambda_1, \lambda_2, \dots, \lambda_s)$ of $|\lambda|=m$ is a sequence of weakly decreasing positive integers summing to $m$. (We may sometimes add or ignore trailing zeroes of $\lambda$ for convenience.)

To any partition, we associate its \emph{Young diagram} (which we also denote by $\lambda$) as follows: in a rectangular grid of boxes, let $(i,j)$ denote the box in the $i$th row from the top and the $j$th column from the left, $i,j \geq 1$. Then the Young diagram associated to $\lambda$ consists of all boxes $(i,j)$ with $j \leq \lambda_i$.

If $\lambda$ and $\mu$ are two partitions such that $\lambda_i \geq \mu_i$ for all $i$, then the \emph{skew Young diagram} $\lambda / \mu$ is the collection of boxes lying in $\lambda$ but not in $\mu$.

We will often assume that our partitions are contained within some $k \times (n-k)$ rectangle (so that $\lambda \subset k \times (n-k)$ means that $\lambda$ has at most $k$ nonzero parts, each of size at most $n-k$). In this case, let us write

\[
\lambda^{\vee}=(n-k-\lambda_k, n-k-\lambda_{k-1}, \dots, n-k-\lambda_1).
\]

Informally, the Young diagrams for $\lambda$ and $\lambda^{\vee}$ exactly fit together to form a $k \times (n-k)$ rectangle if the latter is rotated by a half-turn.

In general, we will refer to any collection of boxes as a \emph{diagram}.

\bigskip

Let $G(k,n)$ denote the Grassmannian of $k$-dimensional subspaces of an $n$-dimensional complex vector space $V$. It is well known that the cohomology ring $H^*(G(k,n))=H^*(G(k,n),\CC)$ can be described as follows. Fix a flag $\{0\}=V_0 \subset V_1 \subset \dots \subset V_n = V$, with $\dim V_i=i$. For $\lambda \subset k \times (n-k)$, we associate the \emph{Schubert variety}

\[
\Sigma_{\lambda}=\{\Lambda \in G(k,n) \mid \dim(\Lambda \cap V_{n-k+i-\lambda_i}) \geq i \text{ for } 1 \leq i \leq k\}.
\]

It is easy to check that $\dim \Sigma_\lambda = k(n-k)-|\lambda|$.

Let $\sigma_{\lambda}$ denote the \emph{Schubert class} in cohomology that is Poincar\'e dual to the homology class of $\Sigma_{\lambda}$. Then $\{\sigma_{\lambda} \mid \lambda \subset k\times (n-k)\}$ forms an additive basis for $H^*(G(k,n))$. Let us write

\[
\sigma_{\mu} \smile \sigma_{\nu} = \sum_{\lambda \subset k \times (n-k)} c_{\mu\nu}^{\lambda} \sigma_\lambda.
\]

The structure constants $c_{\mu\nu}^{\lambda}$ for $H^*(G(k,n))$ are known as the \emph{Littlewood-Richardson coefficients}. They are clearly symmetric in $\mu$ and $\nu$, but they also have a number of other symmetries: in fact, $c_{\mu\nu}^{\lambda}$ is symmetric in $\mu$, $\nu$, and $\lambda^{\vee}$. In particular,

\[
c_{\mu\nu}^{\lambda} = c_{\mu\lambda^{\vee}}^{\nu^{\vee}}.
\]

Since cup product is dual to taking intersections of cycles in homology, it follows that the Littlewood-Richardson coefficients are nonnegative integers. One method of determining these coefficients is given by the classical Littlewood-Richardson rule, which is described below.

A \emph{tableau} is a filling of the boxes in some diagram with positive integers. In the case of a Young diagram, this is said to be a \emph{Young tableau}. By default, Young tableaux are drawn in \emph{English position}, that is, justified to the northwest. At times, it will be convenient for us to consider tableaux justified in different directions. By reflecting $\lambda \subset k \times (n-k)$ across the horizontal axis of the rectangle, we obtain a tableaux justified to the southwest (this is sometimes said to be \emph{French position}). Likewise, a reflection across the vertical axis yields a diagram justified to the northeast, and a half-turn yields a diagram justified to the southeast. We will always number the rows and columns of the rectangle from top to bottom and left to right. The definitions below are to be applied when the tableau in question is drawn in English position.

A (skew) Young tableau of shape $\lambda / \mu$ is said to be \emph{semistandard} if each row is weakly increasing and each column is strictly increasing. The \emph{reverse reading word} of a tableau is the sequence of numbers in the tableau read from top to bottom, right to left. The \emph{weight} of a tableau is the sequence of integers $(\alpha_1, \alpha_2, \dots)$, where $\alpha_i$ is the number of occurrences of $i$.

A \emph{ballot sequence} is a sequence of positive integers such that in any initial segment of the sequence, there are at least as many occurrences of $i$ as there are of $i+1$. A semistandard Young tableau whose reverse reading word is a ballot sequence is called a \emph{Littlewood-Richardson tableau}.

\begin{prop}[Classical Littlewood-Richardson rule]
The coefficient $c_{\mu\nu}^{\lambda}$ is the number of Littlewood-Richardson tableaux of shape $\lambda / \mu$ and weight $\nu$.
\end{prop}

A number of different rules have been described for enumerating Littlewood-Richardson coefficients, all of which are generically called Littlewood-Richardson rules. In the next section we will describe another such rule, which we will later show is in bijection with the classical Littlewood-Richardson rule.

\section{The rule}

We will first describe the algorithm that produces the Littlewood-Richardson coefficients. The algorithm is based upon the Grassmannian algorithm presented by Coskun \cite{Coskun}, which manipulates so-called Mondrian tableaux. Though we do not present an exposition of Mondrian tableaux here, the basic idea is as follows.

The essential information in a Mondrian tableau is a collection of subsets $A_i \subset \{1,2,\dots,n\}$, $1\leq i \leq k$. Let $V_i$ be the subspace of $\CC^n$ spanned by the standard basis vectors $e_j$ for $j \in A_i$. Then we associate the closure of the subvariety of $G(k,n)$ consisting of subspaces $\Lambda$ that have a basis $\{w_1, \dots, w_k\}$ with $w_i \in V_i$. (This is not the definition that Coskun gives, but it is essentially equivalent.) It is the deformation of this subvariety that lies at the heart of the algorithm. Specifically, deforming this subvariety in a specified way yields either a Schubert variety or else a reducible variety whose two components are also represented by Mondrian tableaux. In this way, we can either determine the class of a given tableau or degenerate it into two simpler tableaux. Applying this degeneration iteratively then allows one to determine the class of the original subvariety.

The algorithm we describe below results from a translation of Coskun's Grassmannian algorithm from Mondrian tableaux to Young diagrams with a few combinatorial simplifications. In particular, we transform the subset $A_i$ of $\{1,2,\dots, n\}$ into a subset of $\{1,2,\dots,n-k\}$ of size $|A_i|-k_i$, where $k_i$ is the number of $A_j$ contained in $A_i$. These new subsets then become the rows of our diagram. The degeneration of a Mondrian tableau into two simpler tableaux corresponds to Steps A and B of Algorithm 1. As a result, though we do not present it here, this can be used to give an immediate, geometric proof that Algorithm 1 correctly computes Littlewood-Richardson coefficients. Instead, we will give a purely combinatorial proof of the validity of this algorithm by providing a bijection with the classical Littlewood-Richardson rule.

Let $D=\lambda / \mu$ be a skew Young diagram. Algorithm 1 proceeds by applying several operations to this diagram, eventually resulting in the diagram of a partition $\nu$. However, the algorithm is nondeterministic, so that $\nu$ is not necessarily unique, and there may also be multiple paths to reach a fixed $\nu$. Then $c_{\mu\nu}^{\lambda}$ will count the number of such paths that end at $\nu$.

Algorithm 1 is presented in Table~\ref{alg1}. We denote the leftmost and rightmost boxes in row $i$ by $(i, l_i)$ and $(i, r_i)$. If row $i$ is empty, we will write $l_i = \infty$ and $r_i = 0$. 

\begin{table}
\begin{center}
\begin{tabular}{|p{4.5 in}|}
\hline
\textbf{Input}: A skew Young diagram $D=\lambda / \mu$.\\

\textbf{Output}: A partition $\nu$.\\\\

\textbf{While} there exists $i>1$ such that $l_{i-1}>l_i$, do the following:\\

Take such $i$ to be maximum, and perform either Step A or Step B according to the following rules. If $r_i \geq l_{i-1}-1$ or row $i-1$ is empty, then you may perform Step A, while if $r_i \leq r_{i-1}-1$, then you may perform Step B. (If both conditions hold, choose one to perform.)\\

\textbf{Step A}: For each $l_i \leq j' < l_{i-1}$, switch box $(i,j')$ with box $(i-1,j')$ in $D$.\\

\textbf{Step B}: For each $i'\geq i$ such that $l_i = l_{i'}$, switch box $(i', l_{i'})$ with box $(i', r_{i'}+1)$ in $D$.\\\\

\textbf{Final Step}: Once the $l_i$ are all weakly increasing, shift all boxes in the diagram to the right as far as possible (giving the Young diagram of a partition justified to the northeast).\\

\hline
\end{tabular}
\end{center}
\caption{Algorithm 1}
\label{alg1}
\end{table}

\begin{figure}
\includegraphics[scale=0.8]{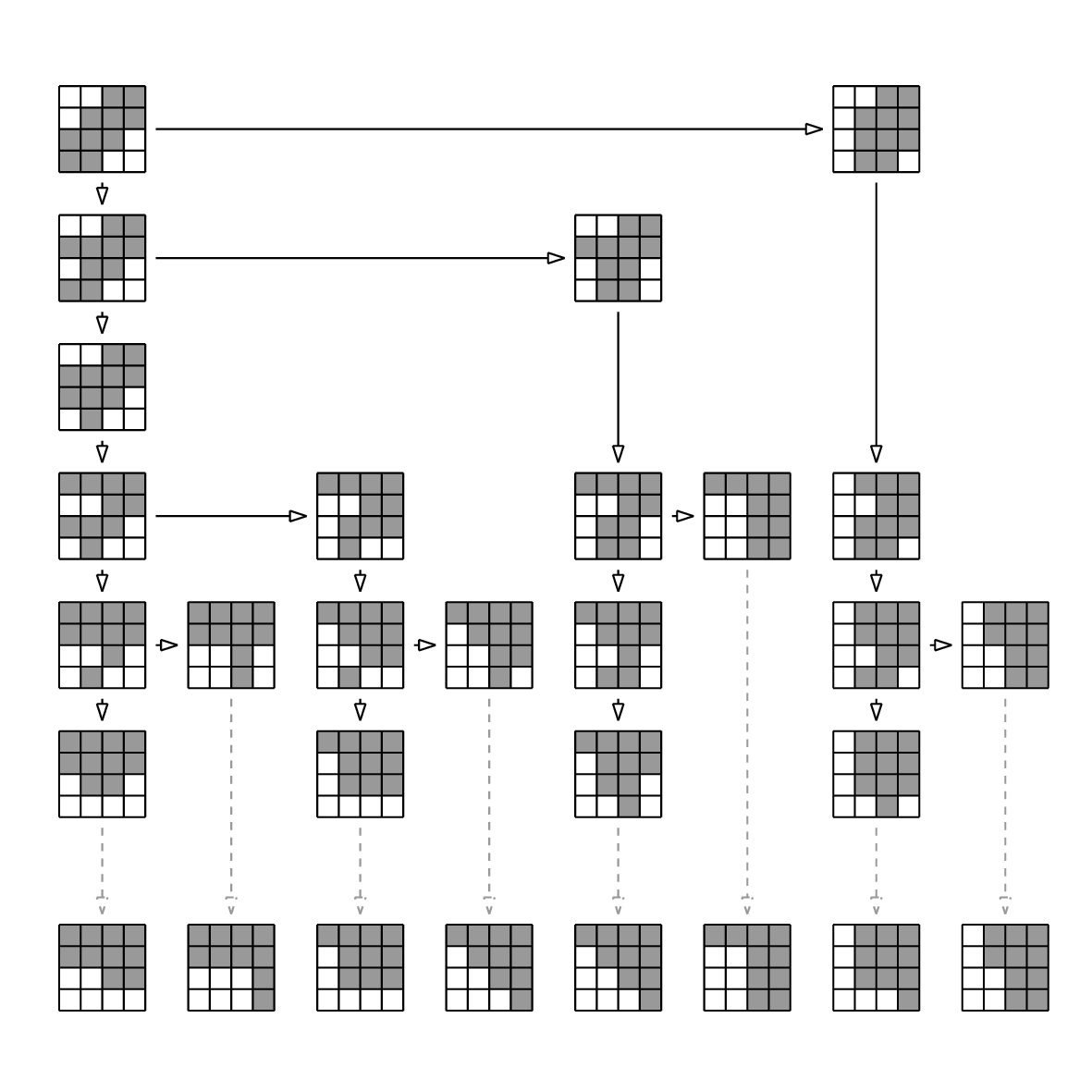}
\caption{Algorithm 1 with $\lambda=(4,4,3,2)$ and $\mu=(2,1)$, showing $c_{\mu\nu}^{\lambda}=2$ for $\nu = (4,3,2,1)$, $c_{\mu\nu}^{\lambda}=1$ for $\nu = (4,4,2), (4,4,1,1), (4,3,3), (4,2,2,2), (3,3,3,1), (3,3,2,2)$, and $c_{\mu\nu}^{\lambda}=0$ otherwise. Equivalently, $\sigma_{21}^2=\sigma_{42}+\sigma_{33}+\sigma_{411}+2\sigma_{321}+\sigma_{222}+\sigma_{3111}+\sigma_{2211}$ in $G(4,8)$. Boxes of $D$ are shaded. Solid vertical arrows indicate applications of Step A and horizontal arrows indicate applications of Step B. The final dotted arrow in each column indicates the final step (if necessary).}
\label{fig1}
\end{figure}

An example computation using Algorithm 1 is given in Figure~\ref{fig1}. Some observations:


\begin{enumerate}

\item \label{1} All intermediate diagrams are row convex, that is, if boxes $(i,j)$ and $(i,j')$ lie in the diagram, then so does $(i,j'')$ for any $j \leq j'' \leq j'$. To see this, note that only Step A can affect row convexity. If row $i-1$ is empty, it simply shifts row $i$ up. Otherwise, the condition on when Step A can be performed ensures that the new row $i-1$ contains boxes in all columns from $l_i$ to $r_{i-1}$, while row $i$ contains boxes in all columns from $l_{i-1}$ to $r_i$.

\item \label{2} If one ignores empty rows, all intermediate diagrams have $r_i$ weakly decreasing. To see this, note that only Step B can affect this condition. This step then replaces $r_{i'}$ with $r_{i'}+1$ if $i' \geq i$ and $l_i = l_{i'}$. By the condition on when Step B can be performed, $r_i+1 \leq r_{i-1}$. Therefore the only problem can occur if we change some row $i'>i$ but not row $i'-1$. This could only happen if $l_{i'-1} \neq l_i = l_{i'}$. But by maximality of $i$, $l_{i'} \geq l_{i'-1} \geq l_i = l_{i'}$, so this is impossible.

\item The algorithm terminates at the Young diagram of a partition justified to the northeast. Indeed, boxes are only moved up or to the right, and no rightmost box is ever moved to the right, implying termination. Since at the end, the $l_i$ are weakly increasing and the $r_i$ are weakly decreasing, it follows that the result is a partition.

\item \label{A-same} By examining the effect of Steps A and B on the $l_i$, one obtains the following: if one ignores empty rows, the sequence of $l_i$ is either of the form
\[l_i \leq l_{i+1} \leq \dots \leq l_k \leq l_{i-1} \leq l_{i-2} \leq \dots \leq l_1\]
or else
\[l_a \leq l_{a+1} \leq \dots \leq l_{i-2} \leq l_i \leq \dots \leq l_k \leq l_{i-1} \leq l_{a-1} \leq l_{a-2} \leq \dots \leq l_1.\] In particular, if Step A is applied to rows $i$ and $i-1$, then any two columns in which a box of $D$ moves are identical above row $i-1$. (One can check that this still holds even if row $i-1$ was empty: if row $i-1$ became empty after applying Step A to rows $i-1$ and $i-2$, then the result follows from the inequalities above. Otherwise it must have been empty in the original diagram, in which case any two affected columns are empty above row $i-1$ anyway.)

\item \label{almostskew} Consider any row convex diagram of boxes with no empty rows such that the $r_i$ are weakly decreasing and the $l_i$ satisfy one of the two inequalities of Observation~\ref{A-same}. We call such a diagram \emph{almost skew}. It is easy to see that any almost skew diagram occurs as an intermediate diagram when applying the algorithm to $\lambda/\mu$, where $\lambda$ has parts of size $r_i$ and $\mu$ has parts of size $l_i-1$.

\setcounter{saveenumi}{\theenumi}

\end{enumerate}


We now come to our main theorem.

\begin{thm} \label{main} The number of ways to apply Algorithm 1 to the skew diagram of shape $\lambda / \mu$ and finish with a diagram of shape $\nu$ is exactly $c_{\mu\nu}^{\lambda}$. \end{thm}

To prove this theorem, we provide a bijection between applications of Algorithm 1 starting with $\lambda / \mu$ and ending at $\nu$ with Littlewood-Richardson tableaux of shape $\nu^{\vee} / \mu$ and weight $\lambda^{\vee}$. Since these tableaux are counted by $c_{\mu\lambda^{\vee}}^{\nu^{\vee}}=c_{\mu\nu}^{\lambda}$, this will prove the result.

Consider the shape $D = \lambda / \mu \subset k \times (n-k)$. Consider the semistandard Young tableau with both shape and weight $\lambda^{\vee}$ justified to the southeast (so all boxes in row $i$ are numbered $k+1-i$). For clarity, let us label the boxes of $D$ with the letter $D$ and keep the boxes of $\mu$ unlabeled.

To deal with the extra labels, we slightly modify the algorithm as given by Algorithm 2 in Table~\ref{alg2}. An example computation using the modified algorithm is given in Figure~\ref{fig2}.

\begin{table}
\begin{center}
\begin{tabular}{|p{4.5 in}|}
\hline
\textbf{Input}: A skew Young diagram $D=\lambda / \mu$ inside a $k \times (n-k)$ rectangle with boxes labeled as described above.\\

\textbf{Output}: A skew tableau of shape $\nu^{\vee}/\mu$ and weight $\lambda^{\vee}$.\\\\

\textbf{While} there exists a box $(i,j)$ that is labeled but $(i-1,j)$ is unlabeled, do the following:\\

Take such $i$ to be maximum, and perform either Step A or Step B according to the following rules. If $r_i \geq l_{i-1}-1$, or if row $i-1$ or row $i$ contains no box labeled $D$, then you may perform Step A, while if $0 \neq r_i \leq r_{i-1}-1$ and $l_i<l_{i-1}$, then you may perform Step B. (If both conditions hold, choose one to perform.)\\

\textbf{Step A}: For all $j'$, if $(i-1,j')$ is unlabeled, switch box $(i,j')$ with box $(i-1,j')$.\\

\textbf{Step B}: For each $i'\geq i$ such that $l_i = l_{i'}$, switch box $(i', l_{i'})$ with box $(i', r_{i'}+1)$.\\\\

\textbf{Final Step}: Once all unlabeled boxes are at the bottoms of their respective columns, shift all boxes labeled $D$ to the right, keeping the rest of the row in order. The numbered boxes then form the desired tableau (in French position).\\

\hline
\end{tabular}
\end{center}
\caption{Algorithm 2}
\label{alg2}
\end{table}

\begin{figure}
\includegraphics[scale=0.8]{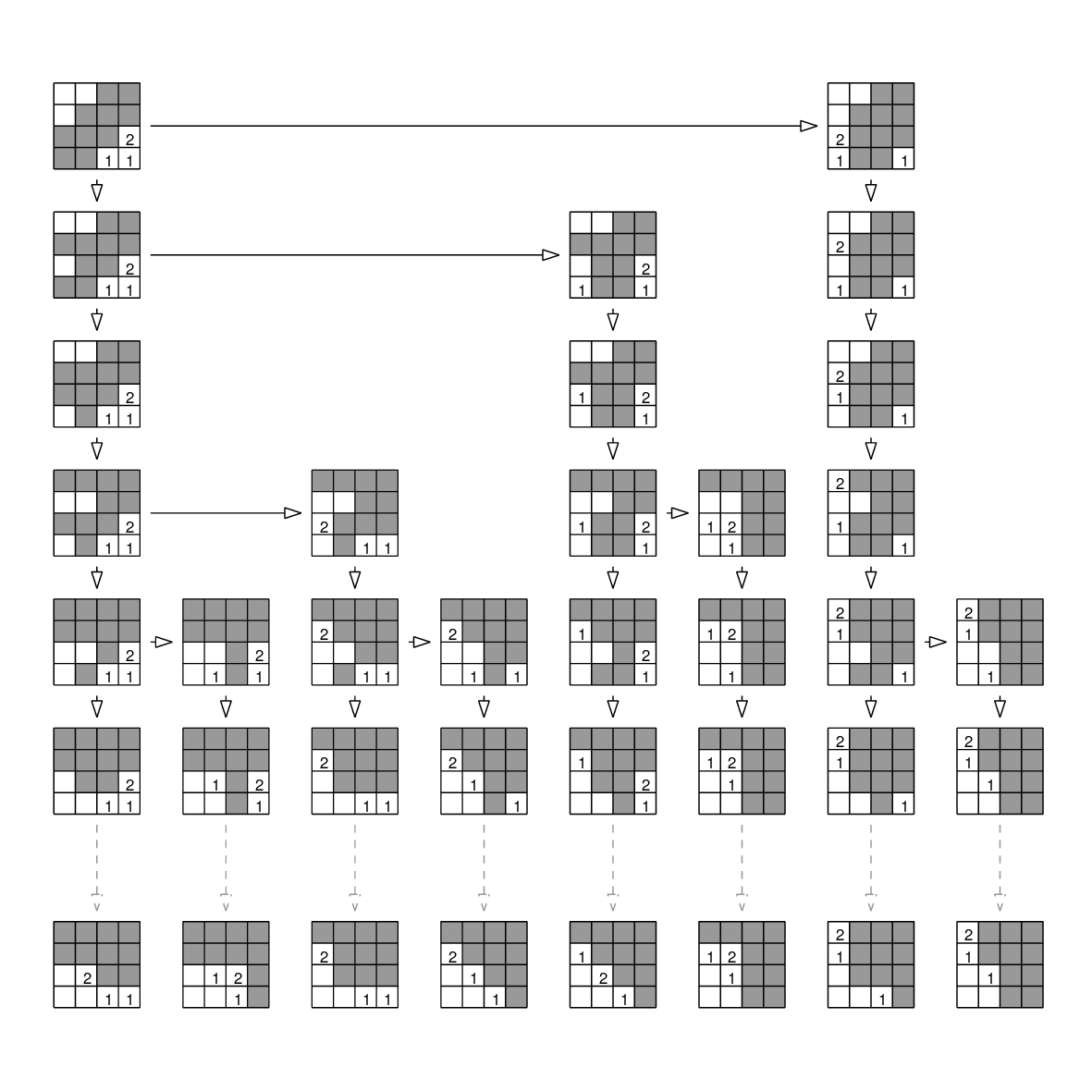}
\caption{Algorithm 2 with $\lambda=(4,4,3,2)$ and $\mu=(2,1)$. For clarity, boxes labeled D are shaded. Solid vertical arrows indicate applications of Step A and horizontal arrows indicate applications of Step B. The final dotted arrow in each column indicates the final step (if necessary). Note the similarity between this figure and Figure~\ref{fig1}.}
\label{fig2}
\end{figure}

Some more observations:

\begin{enumerate}
\setcounter{enumi}{\thesaveenumi}

\item \label{shape} The unlabeled boxes are always left justified within their rows, and they never switch columns. Therefore, at the end of the algorithm, the unlabeled boxes form the same shape $\mu$ that they did at the beginning of the algorithm. It follows that the numbered boxes fill the skew shape $\nu^{\vee} / \mu$.

\item Algorithm 2 acts on $D$ in the same way as Algorithm 1, except that some steps in Algorithm 2 do not move any boxes of $D$ and hence do not exist in Algorithm 1. To see that Step A acts on boxes of $D$ in the same way in both algorithms, it suffices to show that no numbered box ever lies in the same column and above a box labeled $D$, and it is easy to check that this can never happen as a result of performing either Step A or Step B. To see that the extra steps do not change the number of ways to arrive at a diagram where the boxes labeled $D$ have shape $\nu$, it suffices to check that any time we perform Step A in Algorithm 2 without moving any boxes labeled $D$, we could not have performed Step B instead. But in all these cases, either row $i-1$ or row $i$ is empty, or $l_i \geq l_{i-1}$, so only Step A is possible.

\item \label{stepA} Note that the unlabeled boxes are only moved during Step A, and that each such step swaps the unlabeled boxes in two rows. Then notice that (ignoring Step B) the occurrences of Step A are always the same, and they occur in the same order: if $\mu$ has $s$ nonzero parts, then Step A is always performed first for $i=s+1, s+2, \dots$, the last of which produces an intermediate stage $S_s$. Then it is performed for $i=s, s+1, \dots$, resulting in a stage $S_{s-1}$, then for $i=s-1, s, \dots$, resulting in $S_{s-2}$, and so forth. Here, the ellipses indicate that $i$ increases either to $k$ or until rows $i$ and $i+1$ contain the same number of unlabeled boxes. These intermediate stages will be important in Lemma~\ref{mainlemma} below.


\end{enumerate}

We will now show that Algorithm 2 uniquely produces each of our desired tableaux. It will be convenient for us to introduce a slightly different characterization of Littlewood-Richardson tableaux.

\begin{prop} \label{LR-characterization}
Let $T$ be a semistandard skew Young tableau. For any integer $m$, let $T_m$ be the collection of all boxes $(i,j)$ that contain a number greater than $i-m$. Then $T_m$ is a semistandard skew Young tableau. Moreover, $T$ is a Littlewood-Richardson tableau if and only if for every $m$, the weight of $T_m$ is a weakly decreasing sequence.
\end{prop}

\begin{proof}
Suppose $(i,j) \in T \backslash T_m$. Then we cannot have $(i-1,j) \in T_m$, for then $(i,j)$ would contain a number at most $i-m$ but $(i-1,j)$ would contain a number greater than $i-m-1$, which contradicts that the columns of $T$ strictly increase. Likewise, we cannot have $(i,j-1) \in T_m$ since the rows of $T$ weakly increase. It follows that $T_m$ has skew shape and is hence a semistandard skew Young tableau.

Let $T_m$ have weight $(\alpha_{1,m}, \alpha_{2,m}, \dots)$. Then $\alpha_{p,m}$ is the number of occurrences of $p$ in $T$ above row $p+m$. Note that $T$ is a Littlewood-Richardson tableau if and only if the number of occurrences of $p$ above row $i$ is at most the number of occurrences of $p-1$ above row $i-1$. But this condition is just that $\alpha_{p,i-p} \leq \alpha_{p-1,i-p}$, which is exactly the condition that the weight of each $T_m$ is weakly decreasing.
\end{proof}

The key step in the proof of the main theorem is given in the following lemma.

\begin{lemma} \label{mainlemma}
Let $m$ be a positive integer. Suppose that we apply Algorithm 2 to a diagram $D$ of shape $\lambda/\mu$ to arrive at a tableau $T$. Consider the first intermediate stage $S_m$ at which row $p$ contains $\mu_{k+m-p}$ unlabeled boxes for all $k+m-s \leq p \leq k$, where $s$ is the number of nonzero parts of $\mu$. Construct a skew tableau $T^{(m)}$ of some shape $\rho^{(m)}/\mu$ such that the numbers in row $i$ of $T^{(m)}$ are exactly the numbers appearing in the first $\mu_m$ columns of row $k+m-i$ of $S_m$ (in the same order that they appear). Then $T^{(m)}$ is semistandard, as is $T$. Moreover, $T^{(m)}=T \backslash T_m$, where $T_m$ is defined as in Proposition~\ref{LR-characterization}.
\end{lemma}

\begin{proof}
Note first that when $m$ is larger than $s$, $S_m$ is just the initial diagram, and $T^{(m)}$ is empty. Moreover, the other $S_m$ occur immediately after a specific Step A, as is described in Observation~\ref{stepA}. The tableau $T^{(m)}$ is essentially constructed by considering only boxes in the first $\mu_m$ columns and the last $k-m+1$ rows, moving all numbered boxes to the left of boxes labeled $D$ in the first $\mu_m$ columns and reindexing. Since no numbered box appears above a box labeled $D$, it follows that the resulting tableau has skew shape. For convenience, we let $T^{(0)}=T$.

We claim that $T^{(m-1)}$ is obtained from $T^{(m)}$ by adding boxes numbered $i-m+1$ to row $i$. It will then follow by an easy induction that $T^{(m)}$ is semistandard with all numbers in row $i$ at most $i-m$. The claim that $T^{(m)}=T \backslash T_m$ will also follow immediately by the definition of $T_m$.

To prove the claim, note that any numbered box that does not lie in the first $\mu_m$ columns of $S_m$ has never been moved; therefore if it lies in row $i$, then it is numbered $k+1-i$. Note that Step B does not change the row of any numbered box, and any numbered box that ends up in the first $\mu_{m-1}$ columns of $S_{m-1}$ will have been moved up by exactly one row by some occurrence of Step A. Thus a box in row $i$ of $T^{(m-1)}$, which corresponds to a box in row $k+m-i-1$ of $S_{m-1}$, came from row $k+m-i$ in $S_m$. If it came from the first $\mu_m$ columns of $S_m$, then it came from row $i$ of $T^{(m)}$; otherwise, it was numbered $k+1-(k+m-i)=i-m+1$, as desired. (The claim for $m=1$ is essentially the same.)
\end{proof}


It follows immediately from the previous lemma that any output of Algorithm 2 is a Littlewood-Richardson tableau.

\begin{lemma}\label{LR-tableau}
Every Littlewood-Richardson tableau of shape $\nu^{\vee} / \mu$ and weight $\lambda^{\vee}$ is uniquely obtainable from Algorithm 2. \end{lemma}

\begin{proof}
To show that every Littlewood-Richardson tableau $T$ is obtainable, it suffices to show as in the previous lemma that at each intermediate stage $S_m$ corresponding to a tableau $T^{(m)}$ as in the lemma, we can reach the desired intermediate stage $S_{m-1}$ corresponding to a tableau $T^{(m-1)}$. 

Note that for any Littlewood-Richardson tableau $T$, the only boxes in $T^{(m)}$ lie in the first $\mu_m$ columns. (In order for a box $(i,j)$ to contain a number at most $i-m$, there must be at least $m$ boxes of $\mu$ in column $j$.) We also note that $T^{(m)}$ must contain all but possibly $n-k-\mu_{m}$ instances of any number $i$. In other words, there can be at most $n-k-\mu_m$ instances of $i$ appearing above row $i+m$. If $i=1$, this is clear, because the only boxes of $T$ above row $m+1$ lie in the last $n-k-\mu_m$ columns. But if the claim holds for $i$, then it immediately holds for $i+1$ by the ballot word condition, and so the claim holds by induction. 


Now consider $S_m$. To reach $S_{m-1}$, we will perform a number of instances of Step A at rows $m, m+1, \dots$ with some instances of Step B in between. The numbered boxes appearing in $T^{(m-1)}$ will consist of all numbered boxes appearing in the first $\mu_{m-1}$ columns of $S_m$ along with all numbered boxes moved by these instances of Step B. As seen above, the numbers in the first $\mu_{m-1}$ columns must lie in $T^{(m-1)}$.

Algorithm 2 builds $T^{(m-1)}$ from $T^{(m)}$ by adding the columns of $T^{(m-1)} \backslash T^{(m)}$ from left to right. The instances of Step B that move boxes from the last $n-k-\mu_{m}$ columns of the diagram serve to insert these numbers as a column of $T^{(m-1)}\backslash T^{(m)}$; the maximum of the inserted numbers decreases by 1 with each instance of Step A that is performed. (From this, it is clear that Algorithm 2 cannot produce the same tableau in two different ways.) We need to show that when we need to insert a column of $T^{(m-1)}\backslash T^{(m)}$, this is allowed by the condition on when we can perform Step B.

The condition that we can perform Step B on row $i$ when $l_i<l_{i-1}$ and $0 \neq r_i \leq r_{i-1}-1$ means that we can perform the requisite instances of Step B whenever, first, the only boxes in $T^{(m-1)} \backslash T^{(m)}$ lie in the first $\mu_{m-1}$ columns, and second, the weight of $T_m=T\backslash T^{(m)}$ minus the weight of some leftmost columns of $T^{(m-1)}\backslash T^{(m)}=T_m\backslash T_{m-1}$ is weakly decreasing. We have shown above that the first condition always holds.

The second condition also always holds: note that each column of $T_m \backslash T_{m-1}$ contains consecutive numbers and that the maximum number in each column weakly decreases from left to right. Then it suffices to show the following: Let $\sigma$ and $\tau$ are two partitions with $\sigma_i \geq \tau_i$ for all $i$. Suppose that for some $i'$, $\sigma_{i'}>\tau_{i'}$ but $\sigma_{i'+1}=\tau_{i'+1}$. Then for $i'' \leq i'$, 
\[\sigma_1 \geq \dots \geq \sigma_{i''-1} \geq \sigma_{i''}-1 \geq \sigma_{i''+1}-1 \geq \dots \geq \sigma_{i'}-1 \geq \sigma_{i'+1} \geq \dots.\]
But this is obvious from the fact that the parts of $\sigma$ are weakly decreasing, with the only subtlety arising from the fact that $\sigma_{i'}-1 \geq \tau_{i'} \geq \tau_{i'+1}=\sigma_{i'+1}$.

It follows that any Littlewood-Richardson tableau $T$ is (uniquely) obtainable from Algorithm 2, proving the lemma.
\end{proof}

With Lemma~\ref{LR-tableau} proven, the proof of Theorem~\ref{main} is immediate.

\begin{proof}[Proof of Theorem~\ref{main}]
By Lemma~\ref{LR-tableau}, Algorithm 2 to applied to $\lambda/\mu$ uniquely yields every Littlewood-Richardson tableau of shape $\nu^{\vee}/\mu$ and weight $\lambda^{\vee}$ for all $\nu$. Since Algorithms 1 and 2 act on the boxes labeled $D$ in identical ways, the number of ways to apply Algorithm 1 to $\lambda/\mu$ and arrive at a shape $\nu$ is exactly the number of these tableaux, which is $c_{\mu\lambda^{\vee}}^{\nu^{\vee}}=c_{\mu\nu}^{\lambda}$.
\end{proof}

In the next section, we will discuss an application of this result to the theory of Specht modules.

\section{Specht modules}

The Littlewood-Richardson coefficients also play an important role in the study of representations of the symmetric group $\Sigma_n$ on $n$ letters. In this section, we will first construct such representations via Specht modules defined for an arbitrary collection of boxes. (For a reference for the preliminary results in this section, see \cite{Sagan}.) We will then derive a simple result about the structure of Specht modules corresponding to almost skew shapes using the algorithm described in the previous section. 

Consider any diagram $D$ of $n$ boxes. Order the boxes of $D$ arbitrarily, and let $\Sigma_n$ act on them in the obvious way. Let $R_D$ be the subgroup containing those $\sigma \in \Sigma_n$ that stabilize each row of $D$, and likewise define $C_D$ for columns of $D$. Let $\CC[\Sigma_n]$ denote the group algebra over $\Sigma_n$, and consider elements  
\[R(D) = \sum_{\sigma \in R_D} \sigma \mbox{ and }
C(D)=\sum_{\sigma \in C_D} \sgn(\sigma)\sigma.\]
Then we define the \emph{Specht module} $V^D$ (over $\CC$) to be the left ideal
\[V^D=\CC[\Sigma_n]C(D)R(D).\]

It is obvious that this does not depend (up to isomorphism) on the ordering chosen on the boxes of $D$. Note also that permuting rows and permuting columns of $D$ does not change $V^D$ up to isomorphism.

It is well known that the finite-dimensional, irreducible representations of $\Sigma_n$ correspond exactly to $V^{\lambda}$, where $\lambda$ ranges over all partitions of $n$. Moreover, if $\mu$ and $\nu$ are partitions of $m$ and $n$, respectively, then $V^{\mu} \otimes V^{\nu}$ is naturally a representation of $S_m \times S_n$, and this gives an induced representation on $S_{m+n}$, which we denote by $V^{\mu} \circ V^{\nu}$. Then we have the following:

\begin{prop}
\[V^{\mu}\circ V^{\nu}\cong \bigoplus_{|\lambda|=m+n} (V^{\lambda})^{\oplus c^{\lambda}_{\mu\nu}}.\]
\end{prop}

Related to this is the following result: suppose $D$ is the shape of a skew Young diagram $\lambda/\mu$, consisting of $n$ boxes. Then the following proposition holds.

\begin{prop} \label{skew}
\[V^{\lambda/\mu} \cong \bigoplus_{|\nu|=n} (V^{\nu})^{\oplus c^{\lambda}_{\mu\nu}}.\]
\end{prop}

In general, the question of how to decompose the Specht module $V^D$ into irreducible submodules for an arbitrary diagram $D$ is still open. The most general known result, which applies to so-called \emph{percentage-avoiding} diagrams, is due to Reiner and Shimozono \cite{ReinerShimozono}.

\bigskip

Let $D$ be any diagram, and choose two rows $i_1$ and $i_2$. Let us construct another diagram $D^A$ as follows. For $i \neq i_1, i_2$, $(i,j) \in D^A$ if and only if $(i,j) \in D$; $(i_1,j) \in D^A$ if and only if $(i_1,j)$ and $(i_2,j)$ lie in $D$; and $(i_2,j) \in D^A$ if and only if $(i_1,j)$ or $(i_2,j)$ lie in $D$. In other words, to construct $D^A$ from $D$, we move boxes from $(i_1,j)$ to $(i_2,j)$ if possible. This gives an obvious bijection between boxes of $D^A$ and boxes of $D$. If $D$ is an intermediate diagram obtained in Algorithm 1 on which Step A can be applied at row $i$, then $D^A$ with $i_1=i$ and $i_2=i-1$ is the diagram after the step is performed.

The proofs of the following two propositions are adapted directly from James and Peel \cite{JamesPeel}.

\begin{prop} \label{T-definition} There exists a $\CC[\Sigma_n]$-module homomorphism of Specht modules $T\colon V^D \rightarrow V^{D^A}$. \end{prop}
\begin{proof}
We order the boxes of $D$ and $D^A$ to be consistent with the bijection described above. Note $C(D)=C(D^A)$.

Choose a set $X$ of left coset representatives of $R_D \cap R_{D^A}$ in $R_D$ and $Y$ a set of right coset representatives of $R_D \cap R_{D^A}$ in $R_{D^A}$. Let $T$ be the map from $V^D$ to $\CC[\Sigma_n]$ given by right multiplication by $\sum_{\sigma \in Y} \sigma$.

First we show that if $\sigma \in R_D \backslash R_{D^A}$, then $C(D)\sigma R(D^A)=0$. By our choice of $\sigma$, we must have that in $D$, $\sigma$ maps some box $(i_1, j_1)$ to a box $(i_1, j_2)$, with $(i_2, j_1) \not \in D$ but $(i_2, j_2) \in D$. Let $\sigma$ send $(i_2, j_3)$ to $(i_2, j_2)$ in $D$ (and hence in $D^A$ also). But then if $\tau$ is the transposition in $C_D$ that switches boxes $(i_2,j_1)$ and $(i_2,j_2)$, we have that $\sigma^{-1}\tau\sigma$ acts on $D^A$ by transposing $(i_2, j_1)$ and $(i_2, j_3)$, so it lies in $R_{D^A}$. But then
\[C(D)\sigma R(D^A)=C(D)\sigma (\sigma^{-1}\tau\sigma) R(D^A)=(C(D)\tau)\sigma R(D^A)=-C(D)\sigma R(D^A),\]
which implies the claim.

It follows that 
\[C(D^A)R(D^A)=C(D)\left(\sum_{\sigma \in X}\sigma\right) R(D^A)=C(D)R(D)\sum_{\sigma \in Y}\sigma.\] Thus the image of $T$ lies in $V^{D^A}$, as desired.
\end{proof}

In general it seems difficult to precisely determine the kernel of the map $T$ except in special cases \cite{James}. One consequence of Algorithm 1 is that in some cases, this kernel can be naturally described as a Specht module.

If $D$ is an intermediate diagram in Algorithm 1 on which Step B can be performed, let $D^B$ be the resulting diagram. We again order the boxes of $D$ and $D^B$ to be consistent with the bijection implicit in the algorithm.

\begin{prop}
Let $D$ be an intermediate diagram on which both Step A and Step B can be performed at row $i$. Then $V^{D^B}$ is contained in $V^D$ and lies in the kernel of the map $T:V^D \rightarrow V^{D^A}$ described above.
\end{prop}
\begin{proof}
Note $R_D=R_{D^B}$. Note that if during Step B, $(i_1, j_1) \in D$ moves to $(i_1, j_2) \in D^B$, but $(i_2, j_1) \in D \cap D^B$ does not move, then $(i_2, j_2) \in D \cap D^B$. This allows us to use a similar argument as in Proposition~\ref{T-definition}. In other words, let $U$ be a set of right coset representatives of $C_D \cap C_{D^B}$ in $C_D$ and $V$ a set of left coset representatives of $C_D \cap C_{D^B}$ in $C_{D^B}$. Then as in Proposition~\ref{T-definition}, for any $\sigma \in C_D \backslash C_{D^B}$, we have $C(D^B)\sigma R(D)=0$. Then
\[C(D^B)R(D^B) = C(D^B)\left(\sum_{\sigma \in U} \sgn(\sigma)\sigma \right) R(D) = \left(\sum_{\sigma \in V} \sgn(\sigma)\sigma \right) C(D)R(D).\]
It follows that $V^{D^B} \subset V^D$.

To see that it lies in the kernel of $T$, we need that
\[C(D^B)R(D^B)\sum_{\sigma \in Y} \sigma = C(D^B)\left(\sum_{\sigma \in X}\sigma\right) R(D^A) = 0.\] We claim that $C(D^B)\sigma R(D^A)=0$ for all $\sigma \in R_D=R_{D^B}$. As in Proposition~\ref{T-definition}, it suffices to find a transposition $\tau \in C_{D^B}$ such that $\sigma^{-1}\tau\sigma$ lies in $R_{D^A}$.

Note that \[\#\{j\mid (i,j),(i-1,j) \in D\} < \#\{j\mid (i,j),(i-1,j) \in D^B\}.\] Thus there exist $j_1$ and $j_2$ such that $(i,j_1) \in D$, $(i-1,j_1) \not \in D$, both $(i,j_2)$ and $(i-1,j_2)$ lie in $D^B$, and $\sigma$ maps $(i,j_1)$ to $(i,j_2)$. 
Then it is easy to check that letting $\tau$ be the transposition switching $(i,j_2)$ and $(i-1,j_2)$ gives the desired property, proving the result.
\end{proof}

It follows that $V^D$ contains a submodule isomorphic to $V^{D^A} \oplus V^{D^B}$. But in fact, Theorem~\ref{main} allows us to prove a stronger result.

\begin{cor}
Let $D$ be any intermediate diagram occurring in Algorithm 1. If only one of Step A or Step B can be applied to $D$, then $V^D$ does not change after the step is applied. If both Step A and Step B can be applied to D, then $V^{D^B}$ is the kernel of $T: V^D \rightarrow V^{D^A}$. Moreover, $V^D \cong \bigoplus V^{\nu}$, where the sum ranges over all applications of Algorithm 1, starting at $D$ and ending at $\nu$.
\end{cor}
\begin{proof}
Note that by Observation~\ref{A-same}, if only Step A can be applied to $D$, then its only effect is to switch two rows. Likewise, if only Step B can be applied, its only effect is to switch two columns. Since these operations yield isomorphic Specht modules, this proves the first claim.

If both Step A and Step B can be applied, then $V^D$ contains a submodule isomorphic to $V^{D^A} \oplus V^{D^B}$. It follows that $V^D$ contains a submodule isomorphic to $\bigoplus V^{\nu}$ as defined above. But from Proposition~\ref{skew}, we know that equality holds if $D$ is the original skew diagram $\lambda / \mu$. It follows that equality must hold at each intermediate step as well.
\end{proof}

Ignoring empty rows, Observation~\ref{almostskew} characterizes all diagrams that can occur as intermediate diagrams of Algorithm 1. The corollary above thereby allows us to determine the Specht module decomposition of any almost skew diagram.

\section{Conclusion}

In this paper, we have translated Coskun's geometric Littlewood-Richardson rule into a simple combinatorial algorithm that is equivalent to the classical Littlewood-Richardson rule. This suggests that there may be a more direct connection between the geometry of the Grassmannian and the combinatorics of tableaux and representations of $\Sigma_n$.

The algorithm that we have translated here is only a special case of another algorithm that Coskun introduces to compute the Schubert class of more general subvarieties of the Grassmannian, and it would be interesting to see if a similar translation as presented here could yield nontrivial combinatorial results. One could also try to determine whether or not Coskun's algorithm for computing intersections in more general flag varieties is amenable to a similar approach.

Relating to Specht modules, we have shown some special cases in which the map $T$ between Specht modules has a kernel that is also a Specht module. It would be interesting to see if one could show this via a more direct approach or if one could precisely compute the kernel in a more general setting.

\bibliography{lrrule2}
\bibliographystyle{plain}

\end{document}